\documentclass{amsart}
\usepackage{amscd, amssymb, amsmath, amsthm}
\usepackage{amsfonts,enumerate}
\usepackage{latexsym,,txfonts}
\usepackage{graphics,graphicx,psfrag}
\usepackage[colorlinks=true,citebordercolor={0 1 0}]{hyperref}

\newtheorem{theorem}{Theorem}[section]
\newtheorem{lemma}[theorem]{Lemma}

\newtheorem{cor}[theorem]{Corollary}

\theoremstyle{definition}
\newtheorem{definition}[theorem]{Definition}

\theoremstyle{remark}

\def\MR#1{\href{http://www.ams.org/mathscinet-getitem?mr=#1}{MR#1}}
\def\checkMR MR#1#2#3 #4\relax%
  {\ifx#1M%
     \ifx#2R\MR{#3}\else\MR{#1#2#3}\fi
   \else
     \MR{#1#2#3}%
   \fi}

\begin{document}

\title{Chainmail links and $L$-spaces}

\author[Ian Agol]{%
        Ian Agol} 
\address{%
    University of California, Berkeley \\
    970 Evans Hall \#3840 \\
    Berkeley, CA 94720-3840} 
\email{%
     ianagol@math.berkeley.edu}

\thanks{Ian Agol is supported by a Simons Investigator grant  \#376200}

\date{%
June 2023}

\begin{abstract}
In this note we prove that alternating chainmail links are $L$-space links. The proof is inspired by corresponding proofs for double branched covers of alternating links. We also more generally show that flat augmented chainmail links are generalized $L$-space links. Some other properties of these links are also considered. 
 
\end{abstract}

\maketitle


\section{Introduction}
In the process of exploring lens space Dehn fillings on knots, Ozsvath and Szabo introduced the notion of {\it $L$-spaces }\cite[Definition 1.1]{MR2168576} and proved that double branched covers of non-split alternating links are $L$-spaces \cite[Lemma 3.2, Proposition 3.3]{MR2141852}. $L$-spaces are closed rational homology 3-spheres $Y$ whose Heegaard Floer homology is minimal (like a lens space), in the sense that $rank(\widehat{HF}(Y))=|H_1(Y;\mathbb{Z})|$. $L$-spaces are interesting manifolds since they do not admit taut orientable foliations \cite[Theorem 1.4]{OZ0311496}, \cite[Corollary 1.6]{MR3693573}; the converse has been conjectured for irreducible closed 3-manifolds by Juhasz \cite[Conjecture 5]{MR3381327}. Related concepts have been defined for other versions of Floer homology, for example an {\it instanton $L$-space} $Y$ satisfies $dim I^\#(Y) = |H_1(Y;\mathbb{Z})|$ \cite[Definition 1.13]{Baldwin_2022}. It is conjectured that $dimI^\#(Y)=dim \widehat{HF}(Y;\mathbb{C})$ \cite[Conjecture 7.24]{MR2652464}. 

Quasi-alternating links \cite[Definition 3.1]{MR2141852} are a class of links whose double branched covers are $L$-spaces, leading to  the question of whether there are $L$-spaces which are not double branched covers of links? Examples of hyperbolic $L$-spaces were given which do not admit any symmetries, and hence cannot be double branched covers of links \cite[Theorem 1.2]{MR3507256}. Subsequently asymmetric hyperbolic $L$-space knots were discovered leading to further examples \cite{MR4194294}, \cite{anderson2021lspace}. Meanwhile, Boyer-Gordon-Watson conjectured that $L$-spaces do not have orderable fundamental group \cite[Conjecture 1]{MR3072799}, and they proved this for branched double covers of alternating links \cite[Theorem 4]{MR3072799}, with a subsequent simpler proof given by Greene \cite[Theorem 2.1]{MR3778170}. Greene's proof relies on a presentation of the fundamental group associated to a certain Dehn filling on a link described by Ozsvath-Szabo which gives a Kirby diagram for a double branched cover of a link (described in \cite[Section 3.1]{MR3065184}). 

Polyak introduced the notion of chainmail graphs and corresponding chainmail links, such that surgery on this class of links  gives all closed 3-manifolds \cite{Polyak}. He indicated that there exists a finite set of moves on chainmail graphs (analogous to Reidemeister moves) generating the equivalence relation of chainmail graphs giving the same 3-manifold. This is just a restricted class of Kirby diagrams, motivated by an alternative proof of Kirby's theorem given by Matveev-Polyak using a tangle presentation for the mapping class group \cite{MR1266062}. 
The surgery diagrams of Ozsvath-Szabo for double branched covers of links (mentioned in the last paragraph) are certain surgeries on chainmail links. We learned about these topics from a Mathoverflow post by Lucas Culler \cite{315927}.

Gorsky and Nemethi introduced the terminology of $L$-space link (links such that sufficiently large positive surgeries are $L$-spaces) and proved that algebraic links are $L$-space links \cite[Theorem 2]{MR3546454}. Subsequently Yajing Liu explored the properties of $L$-space links and introduced the terminology of {\it generalized $L$-space links} which also admit infinitely many $L$-space fillings \cite[Definition 2.9]{MR3692910}. He asked whether every 3-manifold is a surgery on a generalized $L$-space link \cite[Question 1.19]{MR3692910}? The main result of this paper is that alternating chainmail links are $L$-space links. Thus each double branched cover associated to an alternating diagram of a link fits into infinite families of $L$-spaces. 
 More generally, we prove that fully augmented chainmail links are generalized $L$-space links, hence answering Liu's question positively by Polyak's observation that every 3-manifold is surgery on a chainmail link (and hence on a fully augmented chainmail link). Since the symmetries of flat augmented links are understood, we may also therefore conclude the existence of asymmetric $L$-space links with arbitrarily many components. 

We show that certain positive surgeries on negative alternating chainmail links do not have orderable fundamental group. The proof follows closely the strategy of Greene, using a related presentation of the fundamental group. Finally, we note that all known examples in the literature of $L$-space links have complements that fiber over the circle (this is true for $L$-space knots by \cite{MR2357503}). We observe that alternating chainmail links have fibered complements, continuing this trend. In the final section we ask some natural questions stemming from these examples. 

{\bf Background:} Different sections of this paper assume different background knowledge. Because some of the proofs are fairly short modifications of results in the literature, we will use standard terminology as defined in the primary sources for each section to avoid a plethora of definitions which are only used once. Hopefully the reader will be able to black-box as much as possible and refer to the sources for more details when necessary. 

{\bf Acknowledgements:} We thank Qiuyu Ren, Joshua Greene, Eugene Gorsky, Zhenkun Li and Francesco Lin for helpful comments. 

\section{Chainmail links}

\subsection{Chainmail graphs}

Polyak introduced the terminology of {\it chainmail link} \cite{Polyak} and pointed out that a result of Matveev-Polyak implies that every closed connected orientable 3-manifold is realized by surgery on a chainmail link \cite{MR1266062}. Since Polyak's results have not been published, we review some of his notation and terminology which, although not strictly necessary, will be a useful way to formalize the proofs in this paper. 

\begin{definition} A {\it chainmail graph} is a finite planar graph $G \subset S^2$ whose edges and vertices are decorated with weights in $\mathbb{Z}$. Note that $G$ may have loops and multiedges. Each vertex $v\in V(G)$ is decorated with a weight $\nu(v)\in \mathbb{Z}$, and each edge $e \in E(G)$ is given a weight $\epsilon(e)\in\mathbb{Z}$, where $V(G)$ and $E(G)$ are the set of vertices and edges of $G$ respectively. Thus a chainmail graph consists of a triple $(G,\nu, \epsilon)$, but much of the time we will abbreviate this just as $G$ with the weights implicit (see also \cite{arXiv:2110.13949} for this notation). 
\end{definition}


Given a chainmail graph $G$, one may define a {\it chainmail link} $L_G$ in two steps in the following way. 
For each $v_i\in V(G)$, place a planar unknot $L_i$. Orient $L_i$ counterclockwise. For each edge $e$ connecting vertices $v_i, v_j$, introduce a $\epsilon(e)$-clasped ribbon linking $L_i$ and $L_j$. Let $w_{ij}$ be the sum of edge weights of edges connecting $v_i$ and $v_j$, with the convention that if distinct vertices $v_i$ and $v_j$ are not connected by an edge, then define $w_{ij}=0$. 

\includegraphics[width=4in]{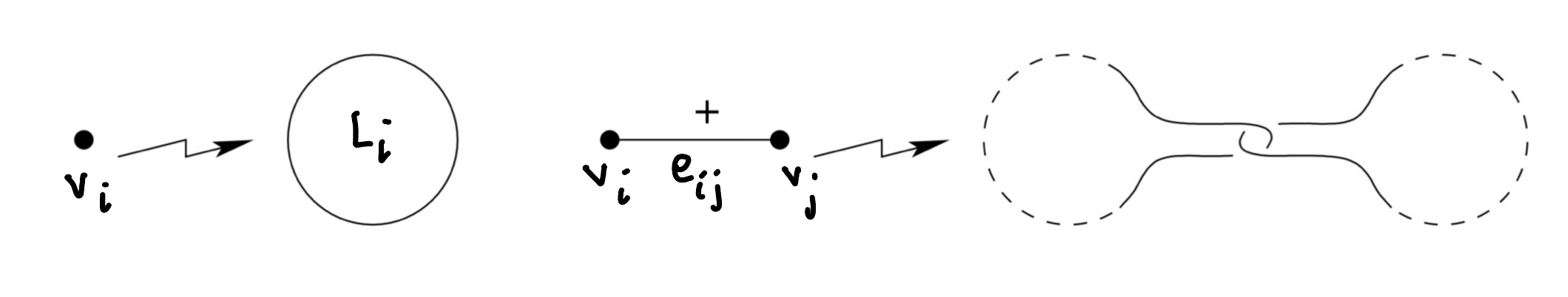}

Now we assign framings to the components $L_i$ with framing $w_{ii}= \nu(v_i)-\sum_{k\neq i} w_{ik}$. This framing is chosen in such a way that the linking number of the framing slope of $L_i$ with the chainmail link is $w_{ii}$. In turn this gives a framing matrix $\Lambda(G) = (w_{ij})$.  Surgery on $L_G$ with these framings gives a 3-manifold $M_G$. Moreover, $det(\Lambda(G))\neq 0$ iff $M_G$ is a rational homology 3-sphere, in which case $|H_1(M_G;\mathbb{Z})|=|det(\Lambda(G))|$. 
We may also view $\Lambda(G)$ as the intersection form of $H_2(W_G)$ of the 4-manifold $W_G$ obtained by adding two handles to the 4-ball along the link $L_G$ with given framings. 

One special property of chainmail links associated to simplicial graphs is that the crossing number is the minimal possible given the linking numbers between components. Chainmail graphs generalize plumbing diagrams for graph manifolds studied by Scharf \cite{MR377918} and Fintushel-Stern \cite{MR595630} who considered the case of trees with edges of weight $\pm 1$ and allowing rational weights on vertices (note however that the vertex labeling convention here may be different). 

We briefly review some of the notation from \cite{MR1266062}.
Consider framed tangles $T_{2n}$ in $\mathbb{R}^2\times [0,1]$ with $2n$ framed endpoints on $\mathbb{R}^2\times \{1\}$ and on $\mathbb{R}^2\times \{0\}$. We may draw these tangles with blackboard framing, and compose the tangles by stacking. 

\begin{figure}[htb]\includegraphics[width=3in]{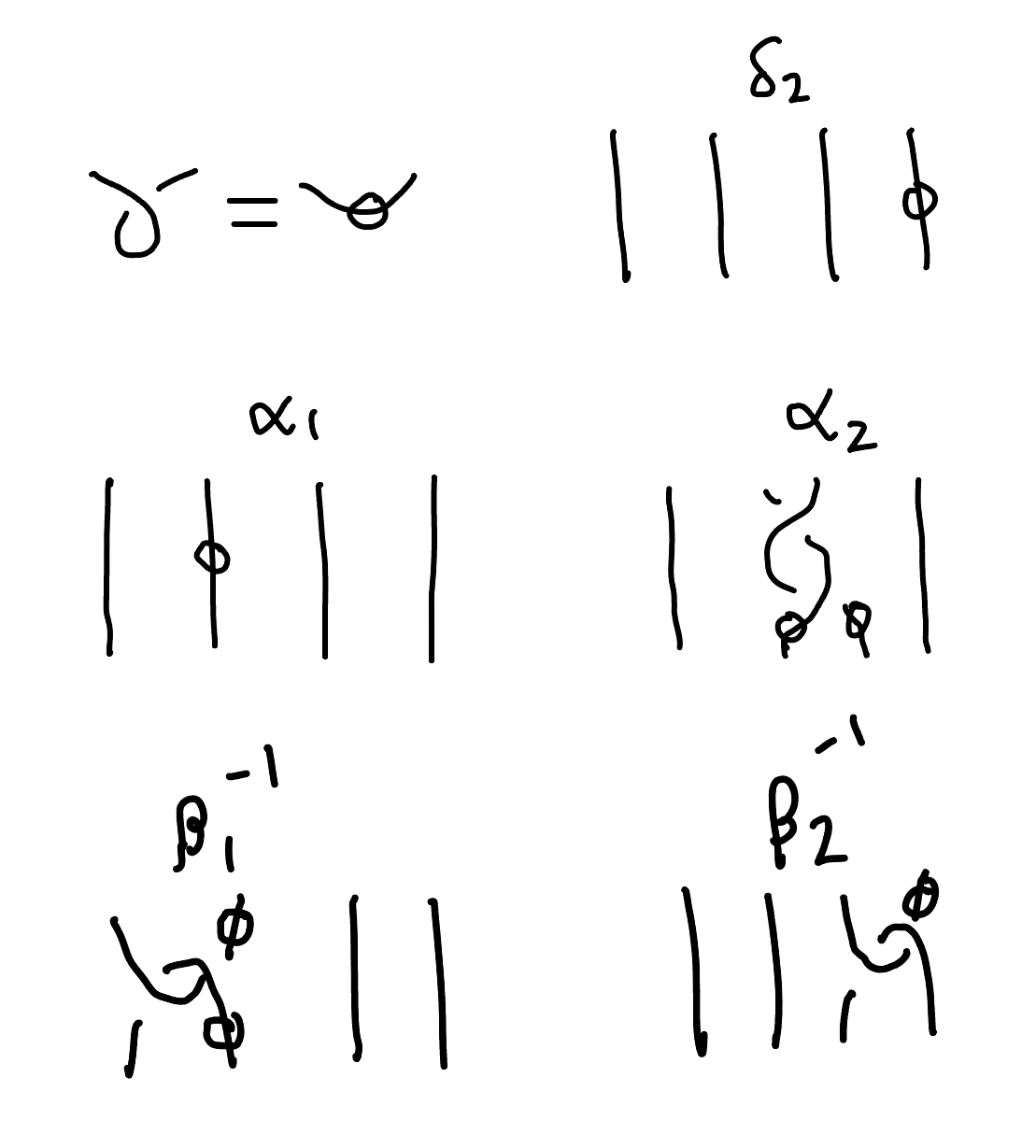}

\caption{Generators for the genus 2 mapping class group} \label{tangles}

\end{figure}

\begin{figure}[htb]\includegraphics[width=2in]{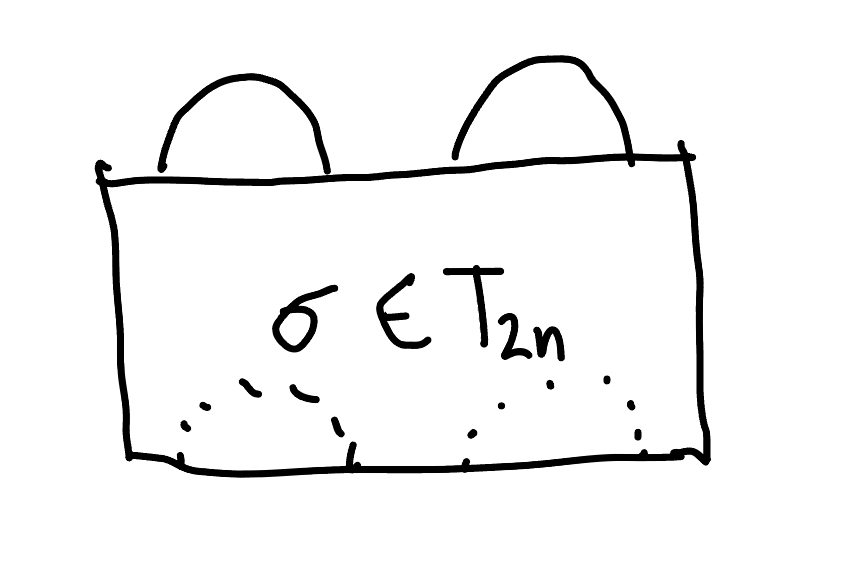}

\caption{Closure of a tangle} \label{closure}

\end{figure}

One may close up a tangle in $T_{2n}$ to get a framed link as in Figure \ref{closure} by adding strands on top of $\mathbb{R}^2\times\{1\}$ and removing the strands meeting the bottom $\mathbb{R}^2\times\{0\}$, and hence a Kirby diagram for a 3-manifold. Matveev and Polyak prove that one may obtain all genus 2 manifolds by stacking the tangles and their mirror images in Figure \ref{tangles}, then closing. In fact they show that these tangles generate the genus two mapping class group when one mods out by Kirby moves, such that the generators in Figure \ref{tangles} correspond to the Dehn twists about curves in Figure \ref{Heegaard}.  For example, the manifold associated to $\beta_2^{-1}\alpha_1\alpha_2\beta_1^{-1}\delta_2\beta_2^{-1}\alpha_1\alpha_2\beta_1^{-1}\beta_2^{-1}$ is shown in Figure \ref{genus2}. 

\begin{figure}[htb]\includegraphics[width=2in]{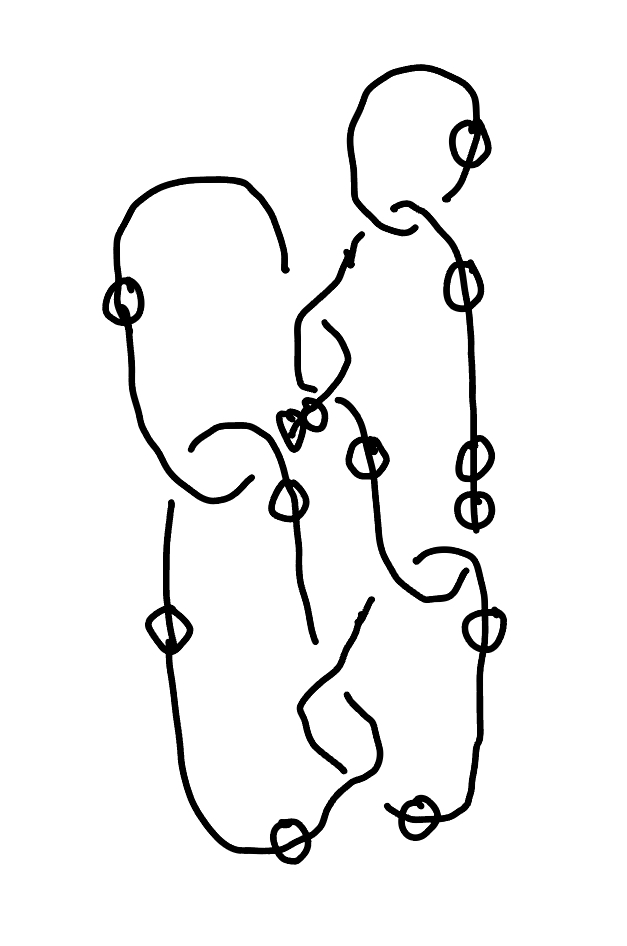}

\caption{A plat closure associated to the tangle $\beta_2^{-1}\alpha_1\alpha_2\beta_1^{-1}\delta_2\beta_2^{-1}\alpha_1\alpha_2\beta_1^{-1}\beta_2^{-1}$} \label{genus2}

\end{figure}

\begin{theorem}[\cite{Polyak}] \label{complete}
Any closed oriented 3-manifold $Y\cong M_G$ for some chainmail graph $G$. 
\end{theorem}
\begin{proof}
The plat representation of a Kirby diagram for a 3-manifold $Y$ given in  \cite[Theorem 7.1]{MR1266062} is easily seen to be an integral surgery on a chainmail link \cite{Polyak} as in Figure \ref{genus2}. 
\end{proof}

We remark on a few moves on chainmail graphs that give rise to the same manifold. These may be verified by using flypes of the corresponding diagram keeping track of framings or Kirby calculus. A zero-weighted edge may be erased. 
Two edges $e_1,e_2$ connecting the same pair of vertices may be redrawn as one edge of weight $\epsilon(e_1)+\epsilon(e_2)$. A loop may be erased. A vertex $v$ with weight $\nu(v)=0$ of degree 1 adjacent to the edge $e$ with $\epsilon(e)=\pm1$ may be removed along with the edge $e$ (see Figure \ref{moves}; see also \cite[Theorem 18.3]{MR817982}). 
Even though we may get rid of loops and multiedges, it is convenient to allow them in chainmail graphs since they can get created under edge contraction. 

\begin{center}
\begin{figure}[htb]\includegraphics[width=4in]{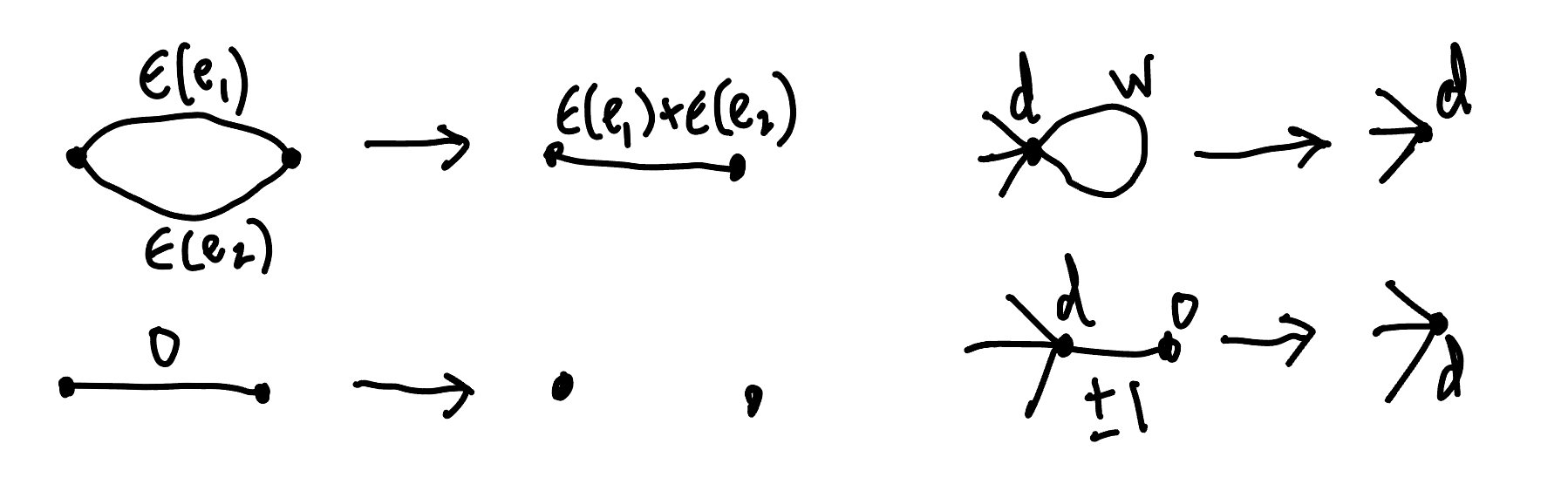}

\caption{Moves on chainmail graphs $G$ preserving the 3-manifold $M_G$} \label{moves}

\end{figure}
\end{center}

\subsection{Alternating chainmail links}

When all of the edge weights of a chainmail graph $G$ are of the same sign, then the associated chainmail link is alternating, see Figure \ref{graph_to_link}.  If the edge weights are all negative, then we call the associated link a {\it negative alternating chainmail link}. If the graph is balanced and connected ($\nu(v)=0$ for all $v\in V(G)$), then it was shown by Ozsvath-Szabo that $M_G$ is homeomorphic to the connected sum of the 2-fold branched cover of the link $K_G$ and $S^2\times S^1$, where $K_G$ is obtained from $G$ as in Figure \ref{medial} \cite[\textsection 3.1]{MR3065184}. 

\begin{center}
\begin{figure}[htb]\includegraphics[width=4in]{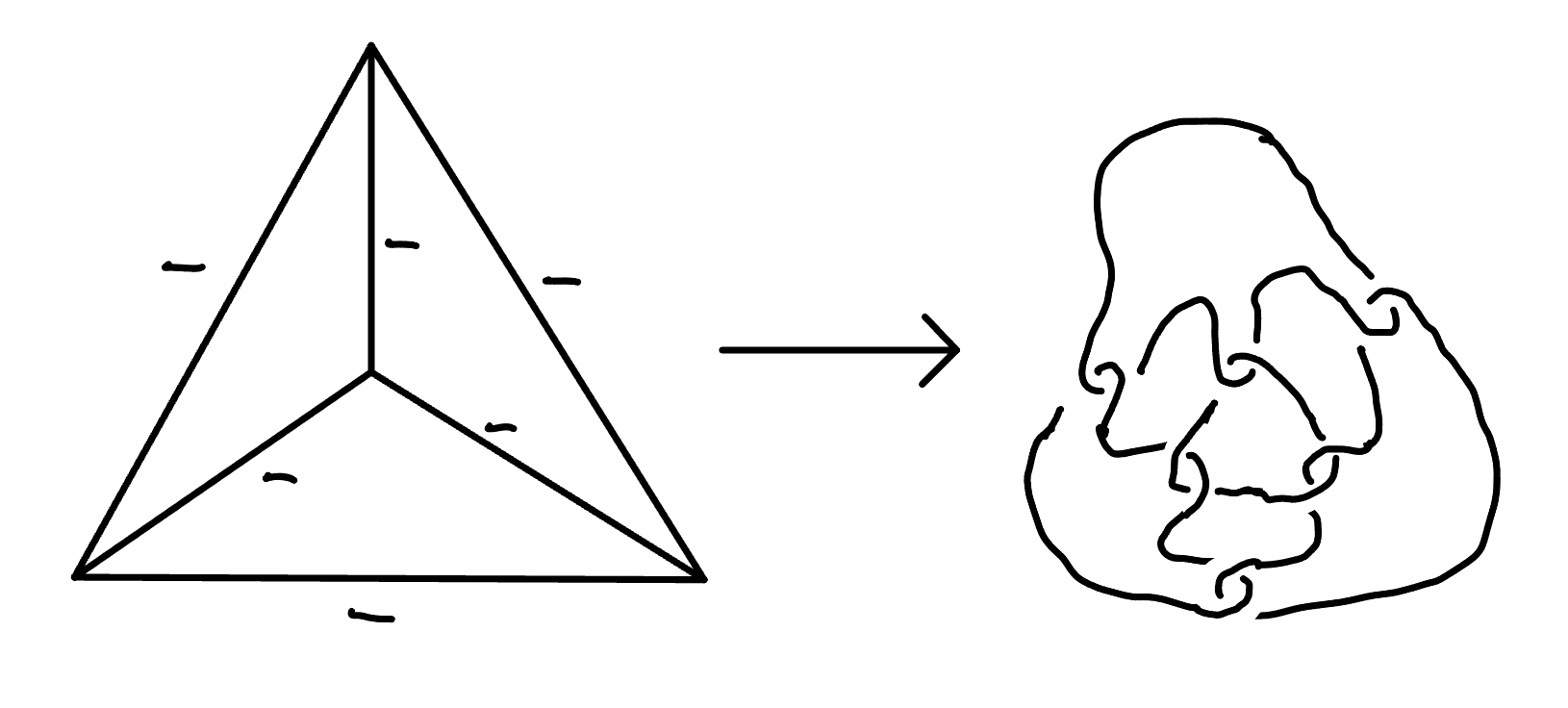}

\caption{From a graph with negative edge weights to an alternating chainmail link} \label{graph_to_link}

\end{figure}
\end{center}

\begin{center}
\begin{figure}[htb]\includegraphics[width=4in]{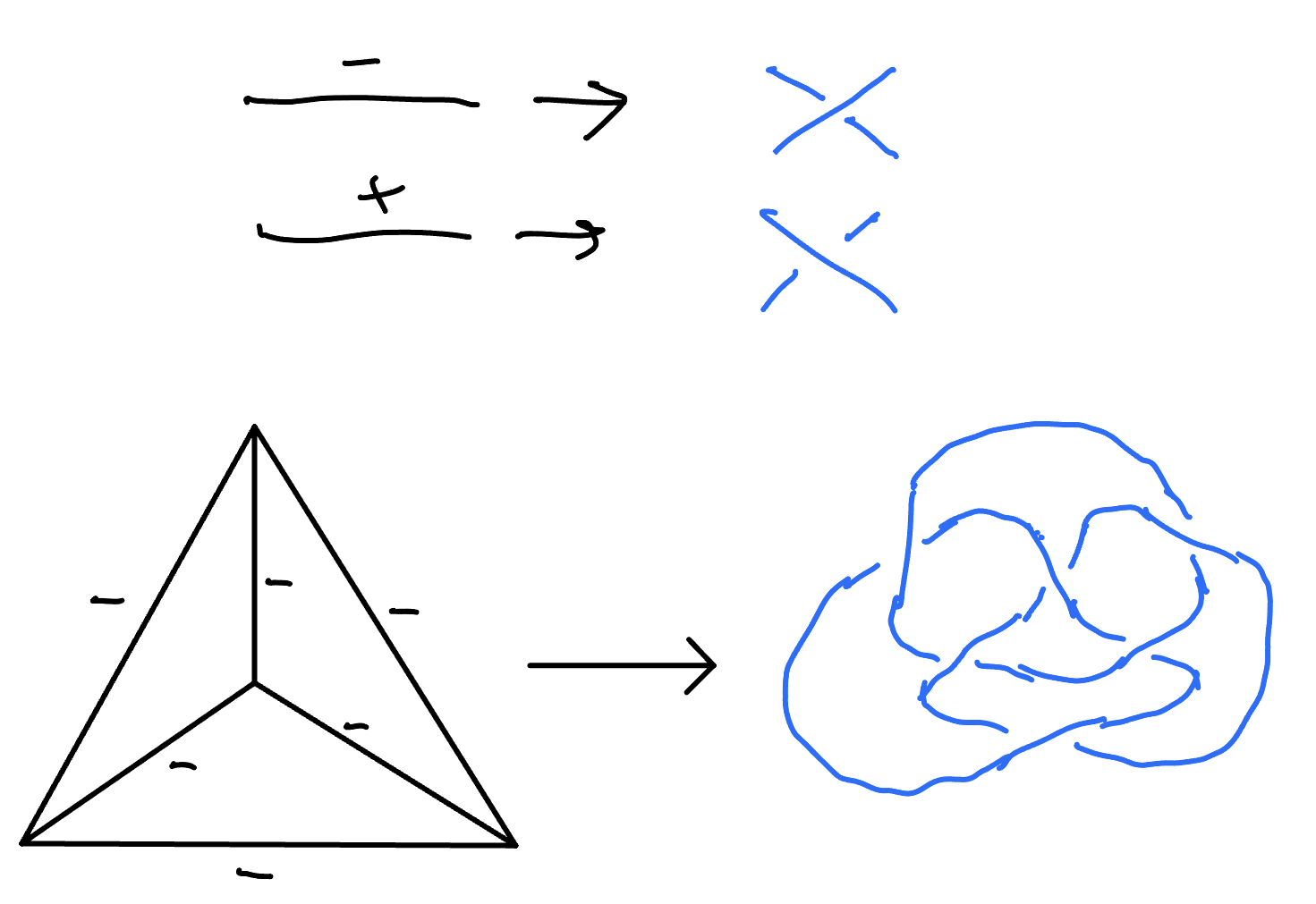}

\caption{From an edge-weighted graph $G$ to a link $K_G$} \label{medial}

\end{figure}
\end{center}

\subsection{Augmented chainmail links}
 
Assume the graph $G$ is simplicial (so that there are no loops or edges with the same endpoints). We may insert crossing circles into $L_G$ then twist to get an augmented link (Figure \ref{augmenting}), possibly at only some of the edges. If we insert crossing circles associated to every edge, then it is called a {\it flat fully augmented (chainmail) link}, otherwise it is a partially augmented link. Thus one obtains a flat augmented link for every planar graph (Figure \ref{graph_to_augmented}).
The labels on a chainmail graph then correspond to framing coefficients for the loops corresponding to vertices and the reciprocals edge weights of the flat fully augmented chainmail link. Thus every 3-manifold is surgery on a fully augmented (flat) link by Theorem \ref{complete}, since every chainmail link is surgery on such a link. 
 
 \begin{center}
\begin{figure}[htb]\includegraphics[width=4in]{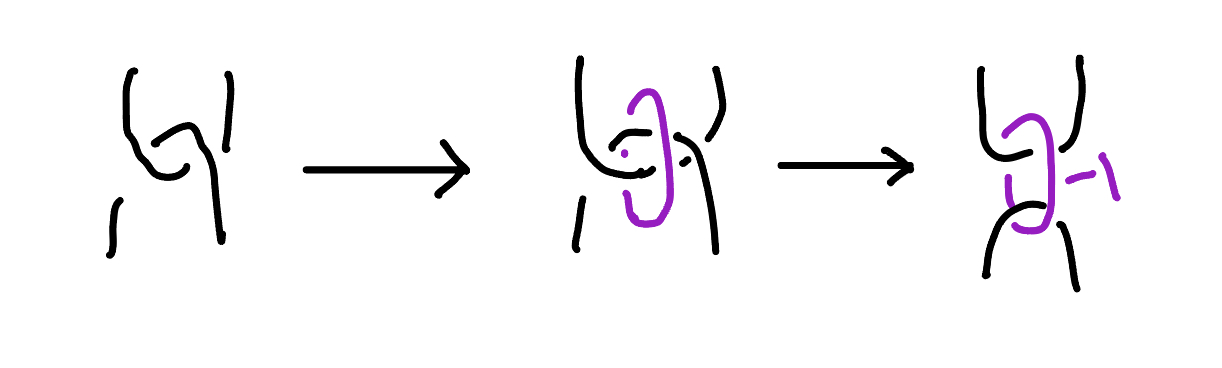}

\caption{Augmenting a chainmail link} \label{augmenting}

\end{figure}
\end{center}
 
\begin{center}
\begin{figure}[htb]\includegraphics[width=4in]{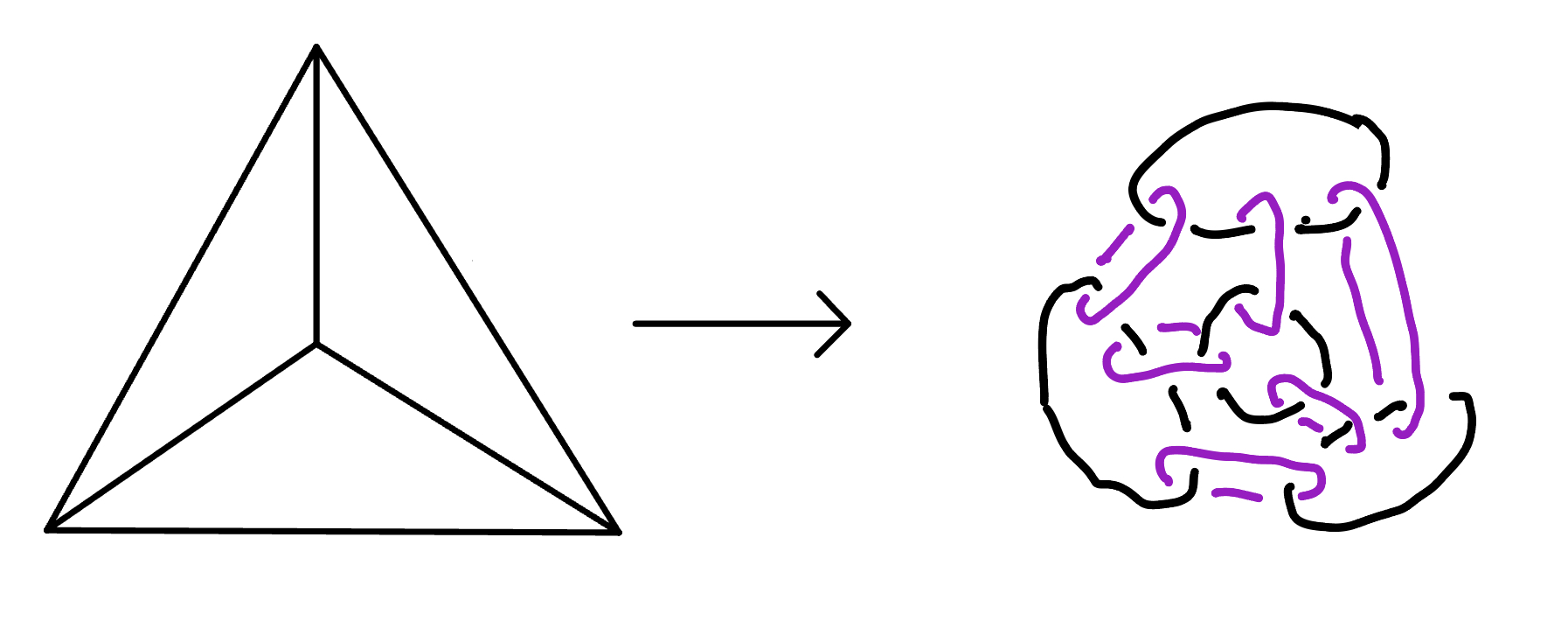}

\caption{Going from a planar graph to a flat augmented link} \label{graph_to_augmented}

\end{figure}
\end{center}
 
The following theorem may be deduced from the main result of \cite{millichap2023flat}:
\begin{theorem}  
A symmetry of a flat fully augmented link complement which sends crossing loops to crossing loops is induced by a symmetry of the link in $S^3$. 
\end{theorem}
This will be used in the proof of Theorem \ref{asymmetric}.
 
\section{Alternating chainmail links are $L$-space links}
In this section, we prove that alternating chainmail links are $L$-space links. Although Theorem \ref{alternating} is a corollary of Theorem \ref{augmented} by taking the coefficients on each crossing loop to be $-1$, we prove it separately as it is serves as a base case of the induction for the more general theorem.  

\begin{definition} \cite[Definition 2.2]{MR3431635}
An $l$-component link $L\subset S^3$ is called an {\it (instanton) $L$-space link} if there exist integers $p_1,\ldots,p_l$ such that $S^3_{n_1,\ldots,n_l}(L)$ is an (instanton) $L$-space for all $n_1,\ldots,n_l$ with $n_i>p_i$ for all $1\leq i\leq l$. 
\end{definition} 

Given a chainmail graph $(G,\nu,\epsilon)$, we may form minors. For an edge $e\in E(G)$ with $\epsilon(e)=-1$, let $G-e$ be the graph obtained by deleting the edge $e$ and keeping the weight functions the same on vertices and edges. If $e$ is not a loop, and $e$ has endpoints $v_1, v_2$, let $G/e$ be the graph obtained by quotienting $e$ to a point (which preserves planarity). The weight function on vertices $V(G)-\{v_1,v_2\}$ is the restriction of $\nu$, and the weight function $\epsilon$ on edges are the same. There is a new vertex $v'$ of $G/e$ obtained by identifying $v_1$ and $v_2$. We define $\nu(v')=\nu(v_1)+\nu(v_2)$ (Figure \ref{minors}).  

\begin{center}
\begin{figure}[htb]\includegraphics[width=4in]{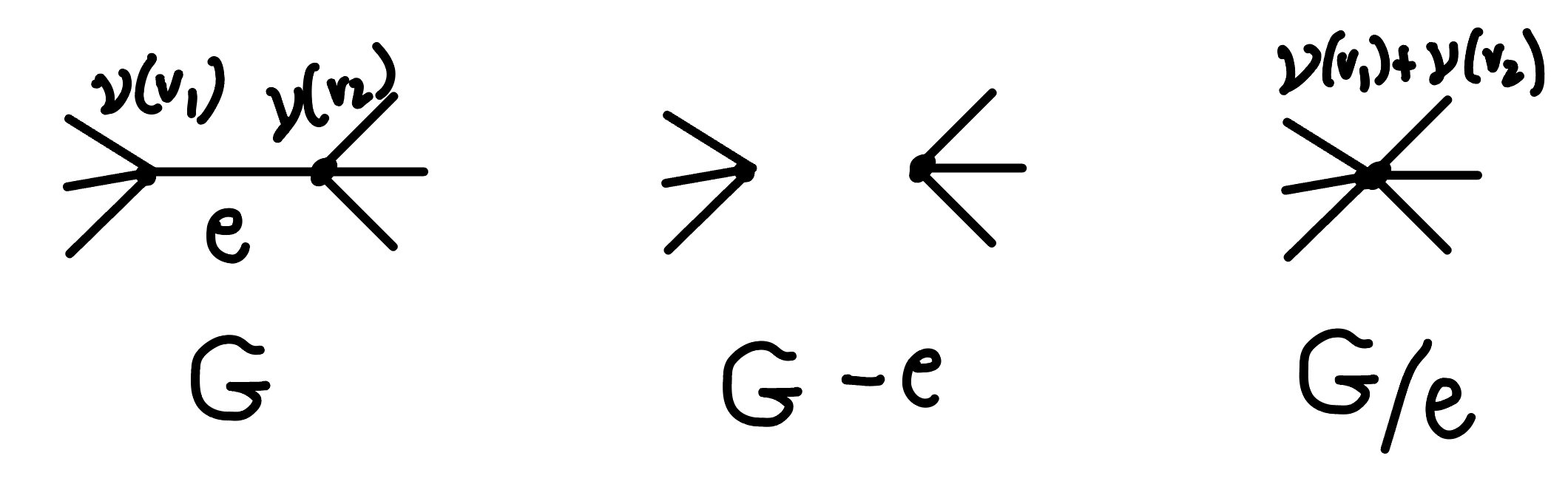}

\caption{Minors of $G$ obtained by deleting and contracting the edge $e$} \label{minors}

\end{figure}
\end{center}

Let us see how the framed links / 3-manifolds associated to the graphs $G, G-e, G/e$ differ, with $\epsilon(e)=-1$. Consider the augmented link obtained by inserting a crossing loop $c$ about $L_1$ and $L_2$, the components corresponding to $v_1$ and $v_2$, then do a twist to remove the double crossing.  Then $L_G$ corresponds to doing $-1$ surgery on $c$, $L_{G-e}$ corresponds to $\infty$ surgery on $c$ (which corresponds to erasing $c$), and $L_{G/e}$ corresponds to $0$ surgery on $c$ (see Figure \ref{crossing_loop}). 
This last claim follows from a Kirby calculus operation $K_3$ described in  \cite[Fig. 3 p. 539] {MR1266062}.

\begin{center}
\begin{figure}[htb]\includegraphics[width=4in]{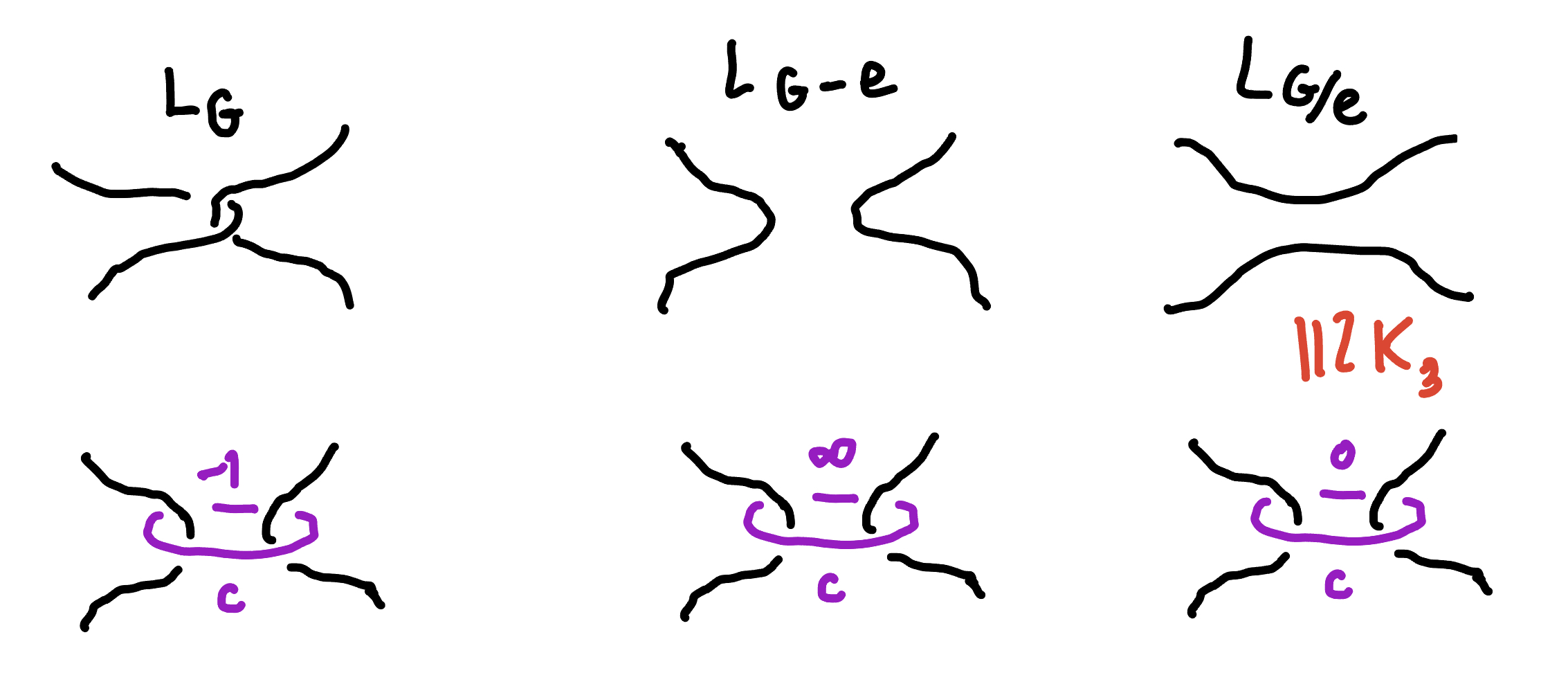}

\caption{Links and surgeries on augmented link corresponding to minors of $G$} \label{crossing_loop}

\end{figure}
\end{center}

Thus we see that $M_G, M_{G-e},$ and $M_{G/e}$ are related by surgeries on the knot $c$ with slopes which have pairwise intersection number $1$. Hence we have an exact triangle $$\widehat{HF}(M_G)\to \widehat{HF}(M_{G-e}) \to \widehat{HF}(M_{G/e}) \to \widehat{HF}(M_G)$$ \cite[Theorem 9.16]{MR2113020}.

The following theorem generalizes \cite[Example 3.15]{MR3692910}

\begin{theorem} \label{alternating}
Negative alternating chainmail links are $L$-space links. More specifically, for a chainmail graph $G$ with all $\nu(v)\geq 0, v\in V(G)$, and at least one $\nu(v)>0$ in each component of $G$, and $\epsilon(e) < 0$ for all $e\in E(G)$,  $M_G$ is an $L$-space and an instanton $L$-space.
\end{theorem}

\begin{proof}
The proof is similar to the proof that double branched covers of (quasi-)alternating links are $L$-spaces in \cite[Proposition 3.3]{MR2168576}. By replacing edges $e$ with weight $\epsilon(e)<-1$ with parallel copies, we may assume that $\epsilon(e)=-1$ for every edge $e\in E(G)$, and we may assume that there are no loops by removing them without changing the manifold. 

If $G$ is a chainmail graph and $\nu(v)=0$ for all vertices and $\epsilon(e)=-1$ for all edges, then the linking matrix $\Lambda(G)$ is the matrix of the graph Laplacian of the underlying unweighted graph, hence all eigenvalues are $\geq 0$, and the $0$-eigenspace has dimension the number of components of $G$ (essentially the combinatorial Hodge theorem for $0$-cocycles). If $G$ has $\nu(v)\geq 0$, and at least one $\nu(v)>0$ in each component of $G$ and $\epsilon(e)=-1$, then when we add the diagonal matrix with diagonal entries $\nu(v_i)$ to the Laplacian matrix to get $\Lambda(G)$, we get a symmetric positive definite matrix. Thus $det(\Lambda(G))>0$ for such graphs.

The proof is by induction on $|E(G)|$, the number of edges of $G$. The base case is an edgeless graph $G$, in which case all  
vertices have positive weight by hypothesis and $M_G$ is a connected sum of lens spaces. 

Now assume that $E(G)=n$, and all such chainmail graphs with fewer edges represent $L$-spaces. Choose an edge $e$ of $G$, and delete the edge to get $G-e$ and contract the edge to get $G/e$, both of which have fewer edges. If there are parallel edges to $e$ (with the same vertices), then $G/e$ may have loops, but we may delete them in order to satisfy the induction hypothesis while still representing the same manifold (see Figure \ref{moves}). 

We may assume that $e$ is not an isthmus, unless $G$ is a tree (equivalently every edge is an isthmus).  If $G$ has at least two vertices with $\nu(v)>0$, then we may choose $e$ so that both components of $G-e$ satisfy the hypothesis that there is at least one vertex $v$ with $\nu(v)>0$. Otherwise, there is a single vertex $v$ with $\nu(v)>0$. Choose any leaf vertex $v'$ with $\nu(v')=0$ (a tree with non-empty edge set must have at least two leaf vertices, so such a vertex exists). Then we may delete $v'$ keeping the same manifold $M_{G-v'}$ by the last move in Figure \ref{moves}. Thus in the latter case $M_G\cong M_{G-v'}$ will be an $L$-space by induction, and thus we may assume we are in the former case in which case $G-e$ is a chainmail graph satisfying the hypotheses. We will assume that the endpoints of $e$ are labeled $v_1$ and $v_2$.

Claim:  $det \Lambda(G)=det \Lambda(G-e) + det \Lambda (G/e)$. 

Proof: Consider $\Lambda(G)$ and $\Lambda(G-e)$ as linear operators on functions on $V(G)=V(G-e)$. We may write the symmetric matrices for these as 

$$\Lambda(G)= \left(\begin{array}{ccc} w_{11} & w_{12} & x_1^T \\ w_{12} & w_{22} & x_2^T \\x_1 & x_2 &  A\end{array}\right) ,   \Lambda(G-e) =  \left(\begin{array}{ccc} w_{11}-1 & w_{12}+1 & x_1^T \\ w_{12}+1 & w_{22}-1 & x_2^T \\x_1 & x_2 &A  \end{array}\right) , $$
differing only in the upper left block corresponding to vertices $v_1$ and $v_2$. Since we are only interested in the determinant, we may do elementary row operations keeping the determinants the same. Add the first row to the second row and first column to the second column (in either order) of $\Lambda(G), \Lambda(G-e)$ to get matrices with the same determinant 

$$\Lambda(G)'= \left(\begin{array}{ccc} w_{11} & w_{11}+w_{12} & x_1^T \\w_{11}+w_{12} & w_{11}+w_{22}+2w_{12} & x_1^T+ x_2^T \\x_1 & x_1+x_2 & A \end{array}\right) , $$  $$\Lambda(G-e)' =  \left(\begin{array}{ccc} w_{11}-1 & w_{11}+w_{12} & x_1^T \\ w_{11}+w_{12} & w_{11}+w_{22}+2w_{12} & x_1^T+ x_2^T \\x_1 & x_1+x_2 & A \end{array}\right) .$$
Then we see that $$\Lambda(G)'-\Lambda(G-e)' = 
 \left(\begin{array}{ccc} 1 & 0 & 0 \\ 0 & 0 & 0 \\ 0 & 0 & 0 \end{array}\right).$$
 By linearity of determinants (say with respect to the first row), we see that 
 $$det(\Lambda(G))-det(\Lambda(G-e)) = \left|\begin{array}{ccc} 1 & 0 &  0 \\ w_{11}+w_{12} & w_{11}+w_{22}+2w_{12} & x_1^T+ x_2^T \\ x_1 & x_1+x_2 & A \end{array}\right| = det (\Lambda(G/e)),$$
since the lower block matrix is $\Lambda(G/e)$. 

This establishes the determinant identity. One may also prove this using \cite[Theorem 4.4]{arXiv:2110.13949}. 

Hence $|H_1(M_G)| = |H_1(M_{G-e})| + |H_1(M_{G/e})|$. By our induction hypothesis, $M_{G-e}$ and $M_{G/e}$ are $L$-spaces, and they are related by an exact triangle. Then by \cite[Proposition 2.1]{MR2168576}, $M_G$ is also an $L$-space, where the knot is the crossing loop $c$. The proof is very similar to the proof of \cite[Lemma 3.2, Proposition 3.3]{MR2141852}. 

For the framed instanton homology case, the base case is the same (lens spaces are instanton $L$-spaces). We have $dim I^\#(Y) \geq |H_1(Y)|$ by \cite[Corollary 1.4]{MR3394316}. There is an exact triangle (with $\lambda =0$) given in \cite[Section 7.5]{MR3394316}. Applied to the case at hand, we get 
$$I^\#(M_G)\to I^\#(M_{G-e};\mu) \to I^\#(M_{G/e}) \to I^\#(M_G)$$ where $\mu$ is a loop representing the core of the Dehn filling of $M_{G-e}$ on the crossing loop $c$ (Figure \ref{crossing_loop}), and similar exact sequences by permuting the role of $\mu$ for $M_G$ and $M_{G/e}$. But in $M_G-c$, two out of the three slopes $\{0, -1, \infty\}$ will be non-trivial in $H_1(M_G-c;\mathbb{F}_2)$ and the third will be trivial by half-lives, half-dies. Filling along a homologically non-trivial slope, the homologically trivial slope will be isotopic to the core, and hence $[\mu]=0\in H_1(M_{G'};\mathbb{F}_2)$, where $G'=G, G-e,$ or $G/e$ depending on the filling slope. So in this case we get an exact triangle 
$$I^\#(M_G)\to I^\#(M_{G-e}) \to I^\#(M_{G/e}) \to I^\#(M_G).$$
Now we apply the same argument as for Heegaard Floer homology to conclude that $$|H_1(M_G)|\leq dim I^\#(M_G) \leq dim I^\#(M_{G-e}) + dim I^\#(M_{G/e}) $$ $$= |H_1(M_{G-e})|+|H_1(M_{G/e})| =| H_1(M_G)|,$$ and hence $M_G$ is an instanton $L$-space. 
\end{proof}

{\bf Remark:} Zhenkun Li pointed out that one can take a connected sum of any component of a negative alternating chainmail link with an $L$-space knot to get another $L$-space link. The point is that the proof goes through as before, with the vertex label adjusted larger, with the base case of the induction including a split link with one component an $L$-space knot so that surgery still gives an $L$-space for large enough coefficients. 

One may give a slightly different characterization of some of the $L$-spaces described in Theorem \ref{alternating} in terms of Heegaard splittings. Consider a set of generators $\alpha_i, \beta_i,\delta_i$ for the mapping class group $Mod(\Sigma)$ of the Heegaard surface $\Sigma$ by right-hand Dehn twists about curves $a_i,b_i,d_i$ respectively shown in Figure \ref{Heegaard}, which represents a standard Heegaard splitting $H_1\cup_\Sigma H_2$ of $S^3$. 

\begin{center}
\begin{figure}[htb]\includegraphics[width=4in]{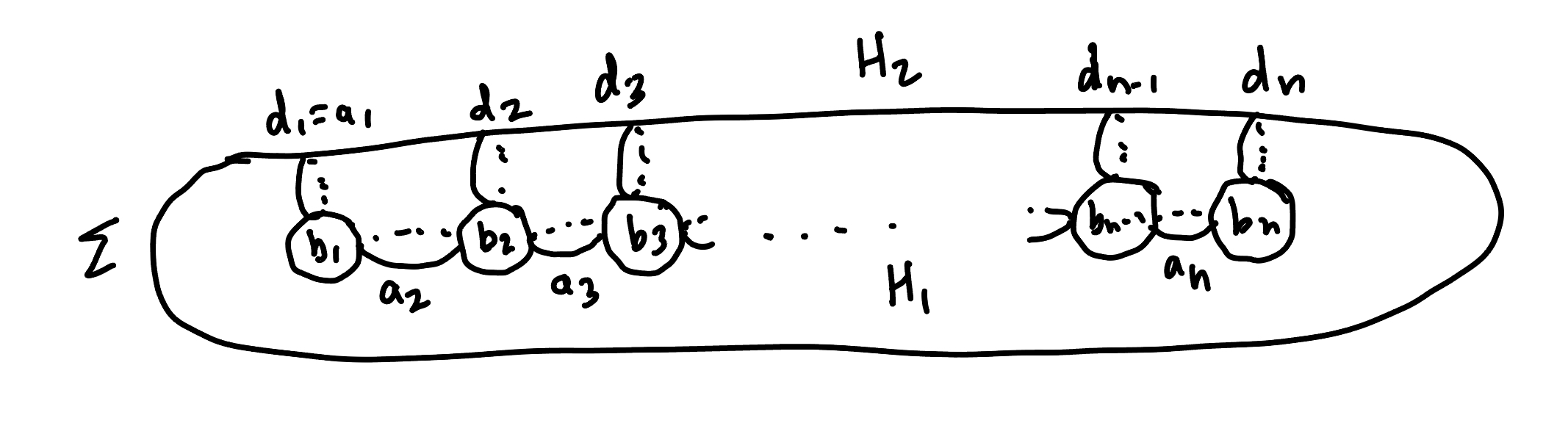}

\caption{Generators for the mapping class group of a Heegaard surface $\Sigma$ for $S^3$} \label{Heegaard}

\end{figure}
\end{center}

\begin{cor} \label{monoid}
Choose $\sigma \in \{\alpha_i,\beta_i^{-1},\delta_i\}^\ast$, the submonoid of products of these elements. Then the resulting manifold $H_1\cup_\sigma H_2$ is a connected sum of an $L$-space and $\#^k S^2\times S^1$ for some $k$.

\end{cor}

\begin{proof}
The Kirby diagram produced from the products of these generators by \cite[Theorem 7.1]{MR1266062} is a negative alternating chainmail link corresponding to a chainmail graph with non-negative vertex weights (see Figure \ref{genus2} for an example). In each component of the chainmail graph with vertex weights $0$, there will be an $S^2\times S^1$ summand connected sum with a double-branched cover of an alternating link in the corresponding component of the Kirby diagram.  Hence the manifold will be of the claimed type in each connect summand, and therefore in total.
\end{proof}

\section{Augmented negative alternating chainmail links are generalized $L$-space links}

Next we consider augmented chainmail links. 

\begin{definition} \cite[Definition 2.9]{MR3692910}
Let $L = \cup_{i=1}^l L_i \subset S^3$ be a link, and for each $i$ choose a sign $\epsilon_i\in \{\pm \}$. 
Then $L$ is called a {\it generalized (instanton) $(\epsilon_1\cdots\epsilon_l)$ $L$-space link} if there exist integers $p_1,\ldots,p_l$ such that $S^3_{\epsilon_1n_1,\ldots, \epsilon_l n_l}(L)$ is a (instanton) $L$-space for all $n_1,\ldots,n_l$ with $n_i>p_i$ for all $1\leq i\leq l$. If we do not specify the sign, then $L$ is denoted a generalized $L$-space link. 
\end{definition}

\begin{theorem} \label{augmented} Partially augmented negative alternating chainmail links are generalized $L$-space links with a $+$ sign associated to each vertex loop and $-$ to each edge (crossing) loop.

\end{theorem}

As a corollary of Theorems \ref{complete} and \ref{augmented}, we answer a question of Yajing Liu \cite[Question 1.9]{MR3692910}. 
\begin{cor} Any closed orientable connected 3-manifold is obtained by surgery on a generalized $L$-space link.
\end{cor}

\begin{proof} \ref{augmented}
First, we need a lemma.

\begin{lemma}
Consider a partially augmented chainmail link simplicial graph $G$, with negative weights on the vertices $V_c(G) \subset V(G)$ corresponding to crossing loops and non-negative weights on the other vertices $V(G)-V_c(G)$. Moreover, the edges between vertices of $V(G)-V_c(G)$ have weight $-1$. Then the sign of the determinant is $(-1)^{|V_c(G)|}$ if non-zero. 

\end{lemma}
\begin{proof}
The proof is by induction on the sum of the negative crossing loop weights. If there are no crossing loops, then $\Lambda(G)$ is is the linking matrix of a negative alternating chainmail graph, and hence the sign of the determinant is non-negative by the proof of Theorem \ref{alternating}, establishing the base case. 

Now consider a single crossing loop of the flat augmented link with weight $- c$, $c > 0$, and its two adjacent loops having framings $a'$ and  $b'$. The matrix $\Lambda(G)$ looks like 

$$\left(\begin{array}{cccc}-c & 1 & -1 & 0 \\1 & a' & 0 & \ast \\-1 & 0 & b' & \ast \\0 & \ast & \ast & A\end{array}\right).$$

When $c=1$, we do column operations adding the first column to the second and subtracting the first column from the third to get 

 $$\left|\begin{array}{cccc}-1 & 0 & 0 & 0 \\1 & a'+1 & -1 & \ast \\-1 & -1 & b'+1 & \ast \\0 & \ast & \ast & A\end{array}\right| = (-1)\left|\begin{array}{ccc}a'+1 & -1 & \ast \\ -1 & b'+1 & \ast \\ \ast & \ast & A\end{array}\right|  .$$

The last determinant is that of a partially augmented chainmail link by doing a $-1$ twist along the crossing loop (see the left diagrams of Figure \ref{crossing_loop}). So the sign of this determinant is $(-1)^{(|V_c(G)|-1)}$ by induction. Thus it is true by induction for $c=1$. 

When $c>1$, we have 

$$\left|\begin{array}{cccc}-c-1 & 1 & -1 & 0 \\1 & a' & 0 & \ast \\-1 & 0 & b' & \ast \\0 & \ast & \ast & A\end{array}\right| =\left|\begin{array}{cccc}-c & 1 & -1 & 0 \\1 & a' & 0 & \ast \\-1 & 0 & b' & \ast \\0 & \ast & \ast & A\end{array}\right| - \left|\begin{array}{cccc}1&0 & 0 & 0 \\1 & a' & 0 & \ast \\-1 & 0 & b' & \ast \\0 & \ast & \ast & A\end{array}\right| .$$
The matrix on the right has determinant 

$$\left|\begin{array}{ccc}  a' & 0 & \ast \\ 0 & b' & \ast \\ \ast & \ast & A\end{array}\right|$$
which is $det \Lambda(G-v)$. By induction, this matrix has the opposite sign, so we get by induction on $c$ that the sign with weight $-c-1$ is $(-1)^{|V_c(G)|}$. 
\end{proof}

Now we apply the exact triangle to see if $M_G$ with $\nu(v)=-c$ is an $L$-space, then $M_G$ with $\nu(v)=-c-1$ is an $L$-space. These are in an exact triangle with the loop erased corresponding to $M_{G-v}$, so we get this by induction again applying \cite[Proposition 2.1]{MR2168576}. The framed instanton case is similar to the argument at the end of the proof of Theorem \ref{alternating}.
\end{proof}

\section{Non-orderable positive fillings on negative alternating chainmail links}

We are not able to show that the $L$-spaces described in Theorem \ref{augmented} have non-orderable fundamental group, addressing \cite[Conjecture 1]{MR3072799}, but we can show non-orderability in special cases. 

\begin{theorem}
For a chainmail graph $G$ with edge weights $-1$, and non-negative vertex weights with at most one vertex weight non-zero, $\pi_1(M_G)$ is non-orderable. 
\end{theorem}

\begin{proof}
The proof is modeled on the proof of Greene that double-branched covers of alternating links have non-orderable fundamental group \cite[Theorem 2.1]{MR3778170} (proved originally by Boyer, Gordon and Watson \cite[Theorem 4]{MR3072799}). 

\begin{center}
\begin{figure}[htb]\includegraphics[width=4in]{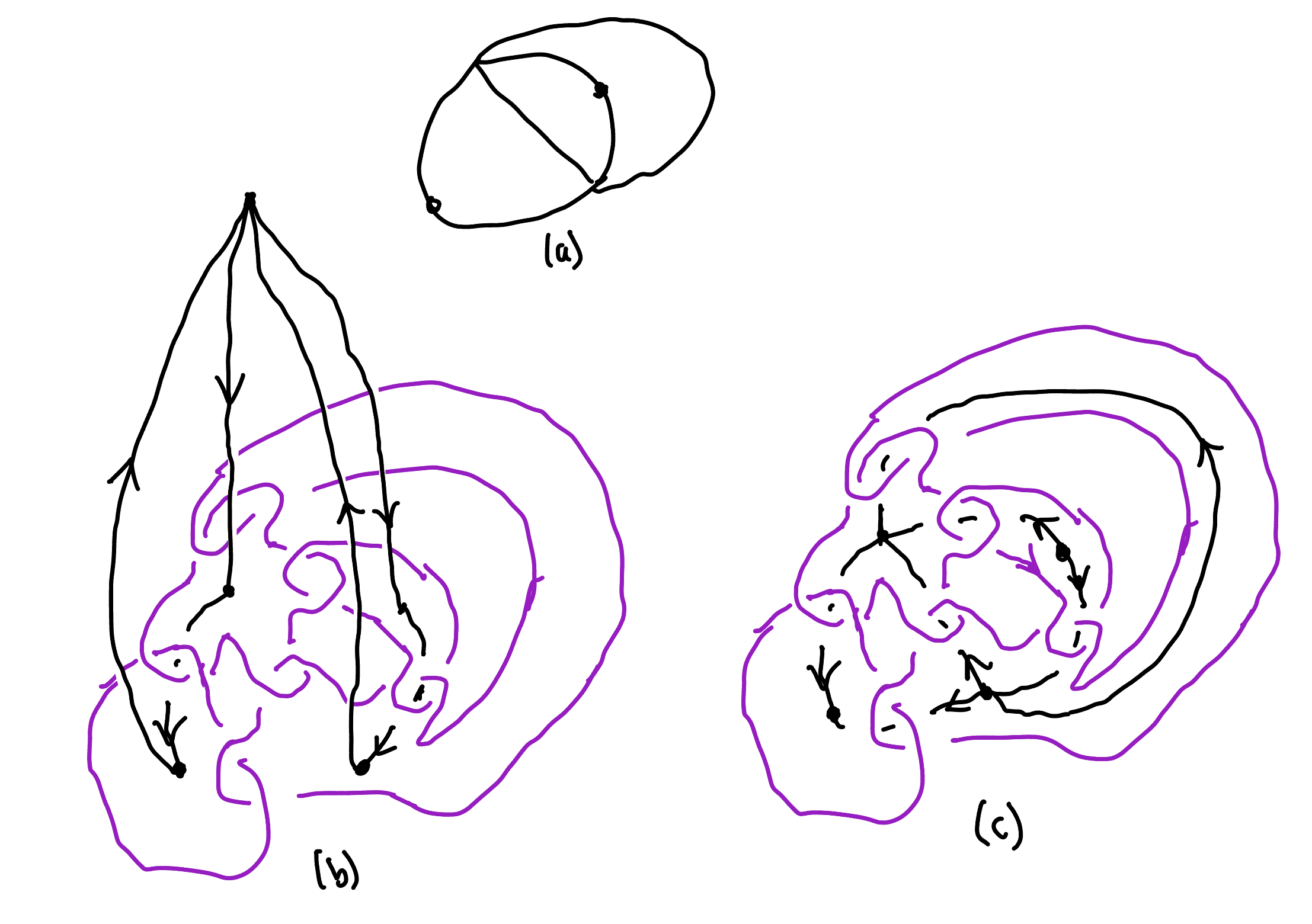}

\caption{(a) Planar graph (b) loops associated to oriented edges (c) acyclic orientation from group ordering} \label{acyclic}

\end{figure}
\end{center}

Each edge of $G$ with a chosen orientation corresponds to an element of $\pi_1(S^3-L_G)$ as shown in Figure \ref{acyclic}(b). If a vertex $v_i$ has weight $\nu(v_i)=0$, then the relation associated to Dehn filling will be the product of the incoming edges (see Figure \ref{zero_framing}). 

\begin{center}
\begin{figure}[htb]\includegraphics[width=4in]{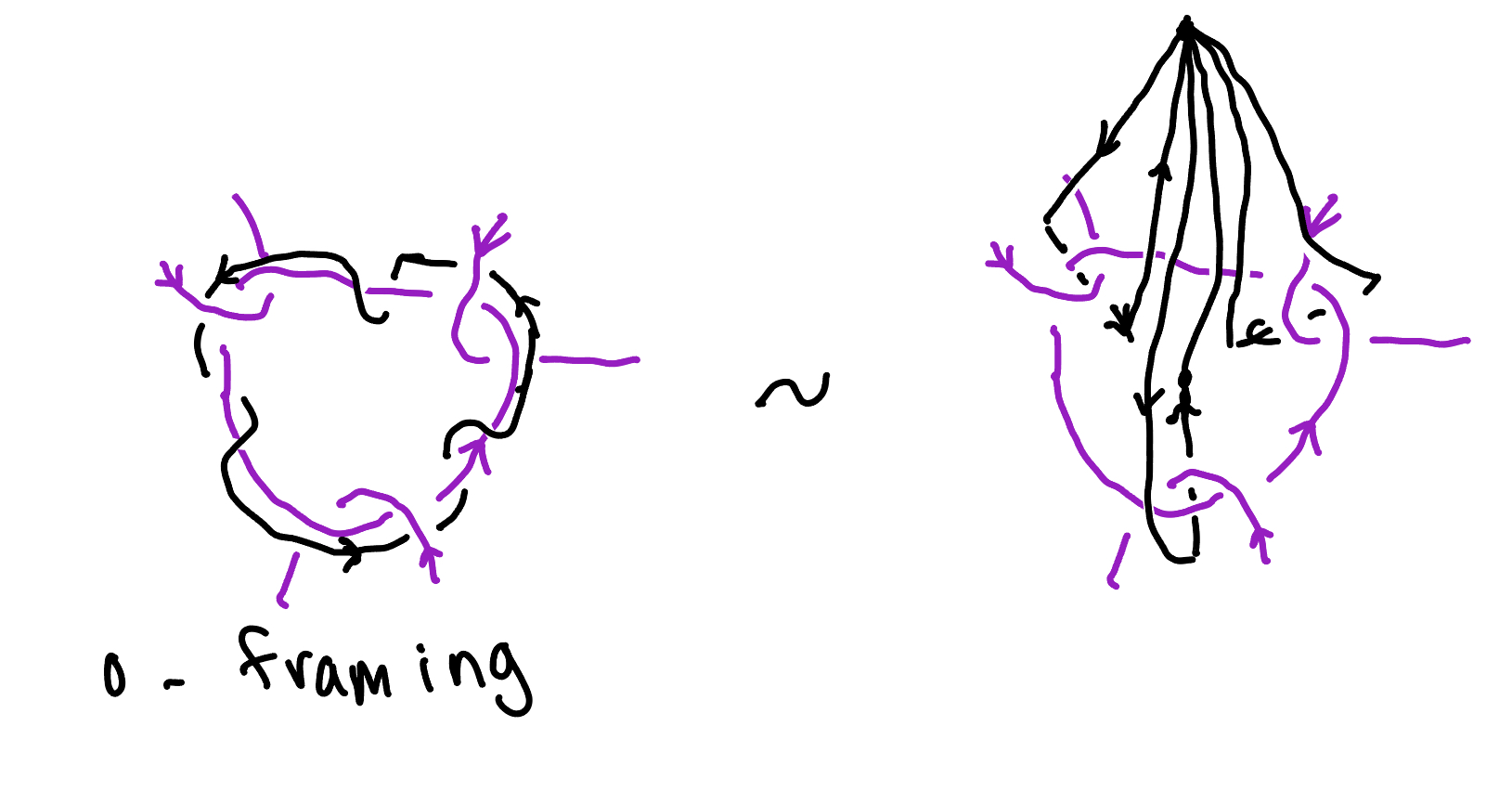}
\caption{The zero-framing loop is a product of incoming edge loops} \label{zero_framing}
\end{figure}
\end{center}

Suppose that $\pi_1(M_G)$ admits a left ordering. Then we may orient each edge so that the associated element is $>1$ (assuming all of these elements are non-trivial first). This gives an acyclic orientation of the graph \ref{acyclic}(c); any oriented loop corresponds to the trivial element but is also a product of elements $>1$, a contradiciton. Then there must be a sink vertex and a source vertex. Since only one vertex $v$ has weight $\nu(v)>0$ by assumption, we may choose a sink or source vertex $v'$ with $\nu(v')=0$. In this case, the product of the edges oriented into this will be trivial but also $>1$ ( or $<1$), a contradiction. 

In the case that some of the edges are trivial, we may still locate a sink or source vertex, allowing that some of the adjacent edges are trivial. If all edges are trivial, then either all vertex weights are zero, and we are in the case considered by Greene \cite{MR3778170}, or one vertex has non-zero weight. In this case the relator from Dehn filling will give a power of the meridian of that loop is trivial. If the meridian is torsion, then the group cannot be orderable. If the meridian is trivial, the one can see that all meridians are trivial (since the edge loops are products of meridians), and hence the fundamental group is trivial. But by definition the trivial group is not left-orderable. 
\end{proof}

A theorem of Tao Li implies that a 3-manifold $Y$ with Heegaard genus 2 that is an $L$-space does not have orderable fundamental group \cite[Theorem 1.2]{li2022taut}. One may apply this to conclude that the manifolds in Corollary \ref{monoid} whose gluing map lies in the monoid $\{\alpha_1, \alpha_2, \beta_1^{-1},\beta_2^{-1},\delta_2\}^*$ do not have orderable fundamental group unless it is a connected sum of copies of $S^1\times S^2$. Such manifolds are realized by surgery on the framed links made by composing the tangles in Figure \ref{tangles}, such as in Figure \ref{genus2}. 

\section{Asymmetric $L$-space links}
It was shown that there exist asymmetric $L$-spaces and asymmetric $L$-space knots \cite[Theorem 1.2]{MR3507256}, \cite{MR4194294}. Here we show the existence of asymmetric $L$-space links. 

\begin{theorem} \label{asymmetric}
There are alternating hyperbolic chainmail links of arbitrarily many components which have trivial symmetry group.
\end{theorem}
\begin{proof}
Start with a flat augmented chainmail link associated to a graph $G$ such as in Figure \ref{graph_to_augmented} with hyperbolic complement (so no parallel crossing loops). Perform $-1/n_i$ surgeries with $n_i$ large on each of the crossing loops $c_i$ to make a negative alternating chainmail link. By choosing the $n_i$ sufficiently large, we may assume that the cores of the Dehn fillings on each of the crossing loops $c_i$ give the $|E(G)|$ shortest loops in the hyperbolic metric on the manifold, and they have distinct lengths \cite[Theorem 5.8.2]{Th}. Thus, any isometry of this link complement must preserve the Margulis tubes about these short loops. Thus it also preserves the complements of these loops which is the initial flat augmented chainmail link complement. By \cite{millichap2023flat}, any homeomorphism of this link complement must be induced by a homeomorphism of the link. Moreover, cusps are taken to cusps, and hence cusps corresponding to crossing loops are fixed. 

\begin{center}
\begin{figure}[htb]\includegraphics[width=4in]{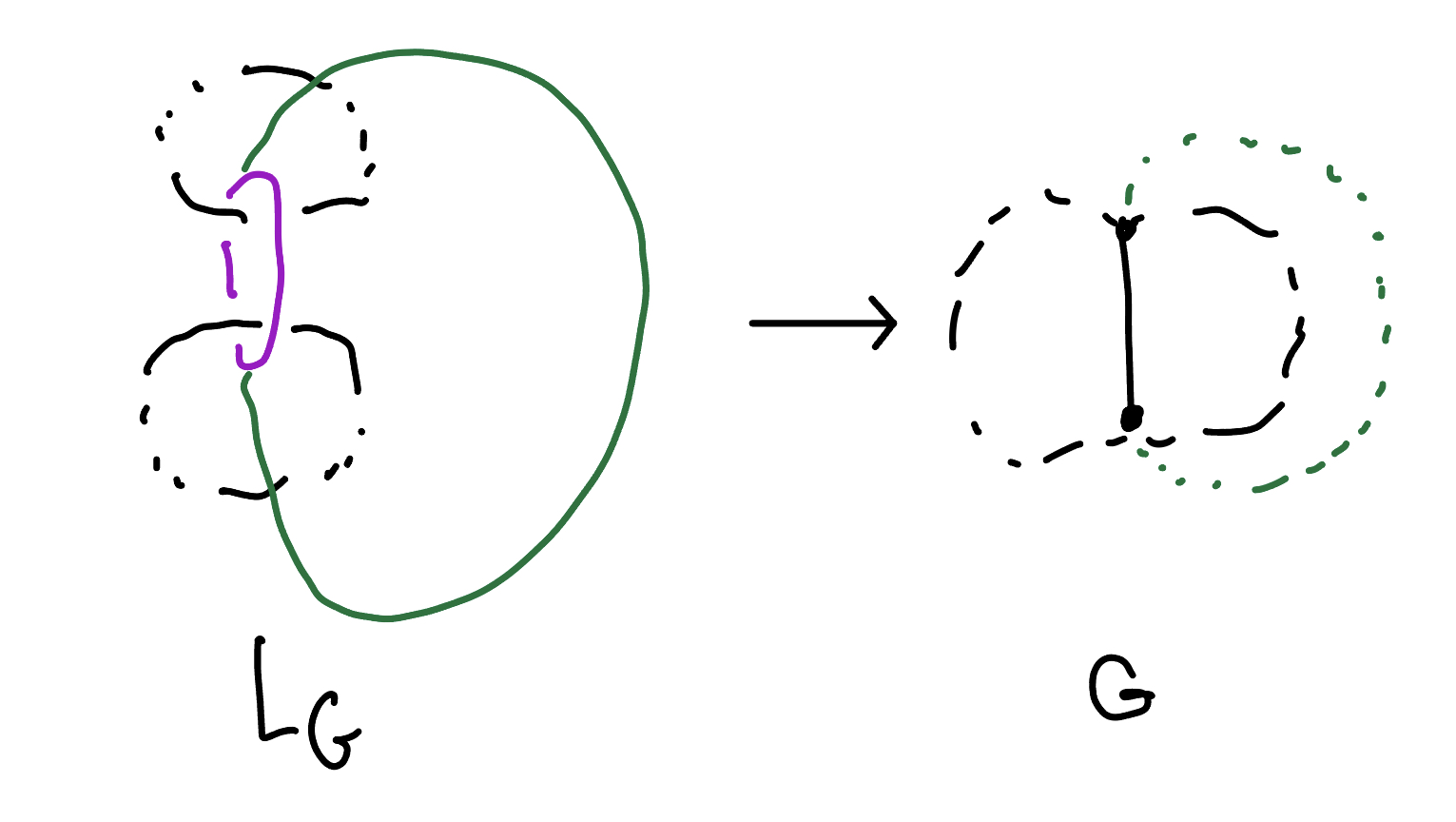}

\caption{A flat augmented chainlink with two crossing disks must have a separating edge } \label{crossingdisk}

\end{figure}
\end{center}

Choose a simplicial planar graph $G$ which contains no triangles and no vertices of degree 2 and no edge that separates the graph.  By \cite[Proposition 2.7]{millichap2023flat}, the only thrice punctured spheres in the  fully augmented chainmail link complement associated to $G$  are crossing disks. Moreover, each crossing circle bounds a unique twice-punctured disk intersecting two knot loops (one can see that for a flat augmented chainmail link that a second crossing circle would give a separating edge, see Figure \ref{crossingdisk}). We may also assume that these have a unique reflection surface, since our assumptions rule out the examples in that paper with multiple reflection surfaces \cite[Theorem 1.2]{millichap2023flat}. Thus, any isometry of $S^3-L$ must send the reflection surface to the reflection surface and the crossing disks to crossing disks. Each loop in the reflection surface bounds a disk punctured by crossing loops and intersecting the crossing disks in arcs. One may recover the graph $G$ from this pattern of intersections, as well as its planar embedding from the reflection surface. So the only orientation preserving symmetries of the link complement must preserve the graph fixing each edge and hence are trivial. There is an orientation reversing symmetry by reflecting in the reflection surface, but this will not preserve the slopes of the fillings of the crossing loops, and hence does not extend to a symmetry of the alternating link. Thus we have shown the existence of negative alternating chainmail links which admit no non-trivial symmetry. 
\end{proof} 

{\bf Remark:} Further Dehn fillings on asymmetric negative alternating chain links give rise to more examples of asymmetric $L$-spaces. Compare to \cite[Theorem 1.2]{MR3507256} and \cite{anderson2021lspace} which require computation to verify the results, and  \cite{MR4194294} which uses an intricate construction to give asymmetric $L$-space knots. 

\section{Alternating chainmail links are fibered}
It was proved that $L$-space knots are  fibered \cite{MR2357503}. Moreover, all of the examples in the literature of $L$-space links have fibered complement. Note that \cite[Example 3.9]{MR3692910} shows that the links do not admit fibered Seifert surfaces; nevertheless the complements fiber. 

\begin{theorem} \label{fibered}
Non-split negative alternating chainmail links have fibered complement.
\end{theorem}
\begin{proof}
Take the checkerboard surface for the alternating chainmail link $L_G$ which is a Seifert surface by applying Seifert's algorithm to the orientation of the link in which all of the loops are oriented counterclockwise (Figure \ref{Seifert}). The complement of this surface is a sutured handlebody $(H,\gamma)$ which Gabai proved is disc-decomposable \cite[Theorem 6.1]{MR780731}. Moreover, one sees that the restriction maps $H^1(S^3-L_G) \to H^1(H)$ is surjective. Choose a cutset of oriented disks $D\subset H$ giving a taut sutured decomposition, and a cohomology class $\alpha'\in H^1(H)$ Poincar\'e dual to $[D,\partial D] \subset H_2(H, \partial H)$. Then there exists a cohomology class $\alpha\in H^1(S^3-L_G)$ so that $\alpha_{|H}=\alpha'$. Let $\mu \in H^1(S^3-L_G)$ denote the cohomology class dual to the Seifert surface. Then $k \mu+ \alpha$ will lie in a cone over a face of the Thurston norm ball adjacent to that containing $\mu$. By  \cite[Theorem 1.4]{AgolZhang} the guts $\Gamma(k\mu+\alpha) = \Gamma(\alpha') = \emptyset$, and hence $k\mu+\alpha$ is a fibered class (see \cite[Definition 2.17]{AgolZhang} for the definition of {\it guts}). 
\end{proof} 

\begin{center}
\begin{figure}[htb]\includegraphics[width=3in]{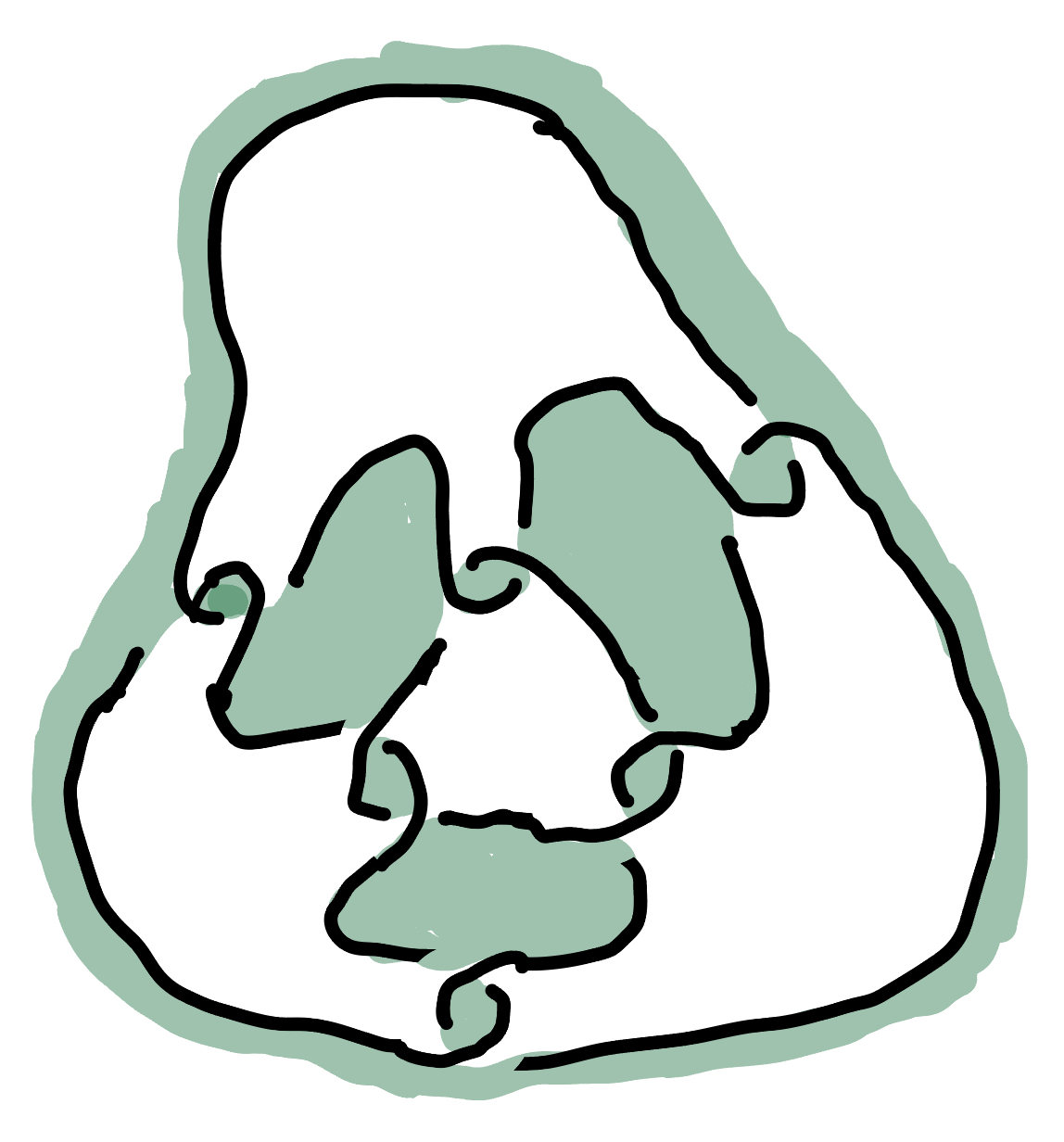}

\caption{A Seifert surface for a chainmail link} \label{Seifert}

\end{figure}
\end{center}

\section{Conclusion}
We conclude with some questions and challenges stemming from the examples in this paper.

\begin{enumerate}
\item
Are any of the new examples of $L$-spaces described in this paper strong $L$-spaces \cite{Greene_2016}? 
\item
Are there non-fibered $L$-space link complements?
\item
Describe all $L$-space surgeries on alternating chainmail and augmented chainmail links. Since this includes all 3-manifolds, such a characterization will include a description of all $L$-spaces and thus is likely to be intricate \cite{GORSKY2018386}. 

\item
Find the foliar surgeries on these links. If the $L$-space conjecture is true, then this should consist of the complement of the reducible and $L$-space fillings \cite[Conjecture 5]{MR3381327}. 

\item
Prove that the $L$-space surgeries on alternating and augmented chainmail links described in these theorems have non-orderable fundamental group (Question (1) is relevant if there is a positive answer by \cite{levine_lewallen_2013}). If the $L$-space conjecture is true, then one would expect the fundamental groups to be non-orderable. 

\item
Presumably one could prove Theorem \ref{fibered} by computing the multi-variable Alexander polynomial of negative alternating chainmail links which allows one to detect if the link complement fibers \cite[Corollary 1.2]{MR2357503}, \cite{MR2507641}, \cite[Theorem 1.2]{MR2393424}. It might be interesting to find a formula for these polynomials and use it to reprove Theorem \ref{fibered}. It would also be interesting to know which augmented alternating links have fibered complement?

\item
Are there more $L$-space links that may be found using this approach? Which chainmail links are (generalized) $L$-space links? Complete this analogy: 

\begin{center} 
alternating links: alternating chainmail links :: quasi-alternating links: ? 
\end{center}
\end{enumerate}

\bibliography{Chainmail_links.bib}
\end{document}